\documentclass[oneside,11pt]{amsart}

\usepackage{amssymb,amsfonts,amsmath,amsthm}
\usepackage{amscd}
\usepackage{tikz-cd}
\usepackage{mathtools}
\usepackage{lmodern}
\usepackage{microtype}
\usepackage{booktabs,array}
\usepackage[all,arc]{xy}
\usepackage{enumerate}
\usepackage{mathrsfs}
\usepackage[toc,page]{appendix}
\usepackage[left=3cm, right=3cm, bottom=3cm]{geometry}
\usepackage{graphicx}
\usepackage{tabularx}
\usepackage{url}
\usepackage{color}
\usepackage{float}
\usepackage{comment}
\usepackage[colorlinks]{hyperref}
\hypersetup{
bookmarksnumbered,
pdfstartview={FitH},
breaklinks=true,
linkcolor=blue,
urlcolor=blue,
citecolor=blue,
bookmarksdepth=2
}
\usepackage[nameinlink,capitalise,noabbrev]{cleveref}

\allowdisplaybreaks

\newtheorem{thm}{Theorem}[section]
\newtheorem*{thm*}{Theorem}
\newtheorem*{cor*}{Corollary}
\newtheorem*{prop*}{Proposition}
\newtheorem{cor}[thm]{Corollary}
\newtheorem{prop}[thm]{Proposition}
\newtheorem{lem}[thm]{Lemma}

\theoremstyle{definition}

\theoremstyle{remark}
\newtheorem{rem}[thm]{Remark}

\makeatletter
\let\origsection\section
\renewcommand\section{\@ifstar{\starsection}{\nostarsection}}

\newcommand\nostarsection[1]
{\sectionprelude\origsection{#1}\sectionpostlude}

\newcommand\starsection[1]
{\sectionprelude\origsection*{#1}\sectionpostlude}

\newcommand\sectionprelude{%
  \vspace{1em}
}

\newcommand\sectionpostlude{%
  \vspace{1em}
}

\makeatother

\makeatletter
\let\c@equation\c@thm
\makeatother
\numberwithin{thm}{section}
\numberwithin{equation}{section}

\crefname{thm}{theorem}{theorems}
\Crefname{thm}{Theorem}{Theorems}
\crefname{prop}{proposition}{propositions}
\Crefname{prop}{Proposition}{Propositions}
\crefname{lem}{lemma}{lemmas}
\Crefname{lem}{Lemma}{Lemmas}
\crefname{cor}{corollary}{corollaries}
\Crefname{cor}{Corollary}{Corollaries}
\crefname{defn}{definition}{definitions}
\Crefname{defn}{Definition}{Definitions}
\crefname{rem}{remark}{remarks}
\Crefname{rem}{Remark}{Remarks}

\title[Bruzzo--Gra\~na Otero's curve semistability conjecture]
{A Counterexample to the Curve Semistability Conjecture for Higgs Bundles}
\author{Pengfei Huang}


\begin{document}

\pagenumbering{arabic}

\begin{abstract}
In this paper, we give a counterexample to the Bruzzo--Gra\~na Otero conjecture on the curve semistability for Higgs bundles.  Let $\Sigma$ be a very general smooth plane quintic curve and let $X=\Sigma^{(2)}$ be its second symmetric product. Starting from the tautological bundle $F=\mathcal{O}_\Sigma(1)^{[2]}$, we construct a rank four Higgs bundle $\mathcal{E}_s=(E_s,\theta_s)$ on $X$.  Then, for every smooth projective curve $C$ and every morphism $f:C\to X$, the pullback Higgs bundle $f^*\mathcal{E}_s=(f^*E_s,f^*\theta_s)$ is semistable.  However, on the other hand, $\mathrm{det}(E_s)\cong\mathcal{O}_X$ and $\int_Xc_2(E_s)=10$. Consequently, the discriminant of $E_s$ does not vanish, and $\mathcal{E}_s$ provides the desired counterexample.
\end{abstract}

\maketitle

\section{Introduction}\label{sec:introduction}

Let $X$ be a smooth complex projective variety and let
$\mathcal{E}=(E,\theta)$ be a Higgs bundle of rank $r$ on $X$.  The discriminant of $E$ is defined by
\[
 \Delta(E)=c_2(E)-\frac{r-1}{2r}c_1(E)^2.
\]
The Bruzzo--Gra\~na Otero conjecture asserts that the following two conditions are equivalent:
\begin{enumerate}
\item $\mathcal{E}$ is semistable with respect to some polarization and $\Delta(E)=0$.
\item For every smooth projective curve $C$ and every morphism
      $f:C\to X$, the pullback Higgs bundle $f^*\mathcal{E}$ is semistable.
\end{enumerate}
Here and throughout this paper, if $f:C\to X$ is a morphism from a smooth projective curve and
$\varphi: E\to F\otimes\Omega_X^1$ is an $\Omega_X^1$-valued morphism of vector bundles, we denote by $f^*\varphi$ the contracted pullback
\begin{align}\label{eq:pullback-higgs-field}
 f^*E\xrightarrow{\,f^*(\varphi)\,}
 f^*F\otimes f^*\Omega_X^1
 \xrightarrow{\,1\otimes df\,}f^*F\otimes K_C.
\end{align}
In particular, $f^*\theta$ denotes the pullback Higgs field, and $f^*\mathcal{E}=(f^*E,f^*\theta)$ denotes the pullback Higgs bundle.  

Simpson called condition~(2) pluri-semistability for smooth proper Deligne--Mumford stacks and the corresponding statement on the vanishing of  all Chern classes the ``Bruzzo conjecture''
\cite[Definition~9.2 and Conjecture~9.5]{Sim11}.  On smooth projective varieties, the curve conditions agree, while the conclusions differ in general but agree on surfaces with trivial determinant, as in our example.
The twist-invariant discriminant formulation was later stated explicitly by Biswas, Bruzzo and Gurjar \cite[Conjecture~1.2]{BBG19}.  It was called Bruzzo's conjecture by Lanza and Lo Giudice \cite{LLG17}, whereas a recent overview uses the name Bruzzo--Gra\~na Otero conjecture \cite[Conjecture~1]{Cap26a}.  We follow this terminology and refer to the implication from (2) to (1) as Bruzzo--Gra\~na Otero's curve semistability conjecture, or simply the Bruzzo--Gra\~na Otero conjecture.  The implication from (1) to (2) was proved in \cite[Theorem~4.7]{BGO11} and also follows from Simpson's pullback-compatible tensor equivalence \cite[Corollary~3.10 and the discussion following it]{Sim92}, applied to the endomorphism Higgs bundle.  The converse is valid in rank two \cite[Theorem~4.8]{BGHR23}, and it is also known for several classes of varieties and under additional assumptions on Higgs Grassmannians \cite{BLG16,BLLG19,BP22,BC23}.  Recent works still formulate the general implication as an open problem \cite[Section~2]{BCGO26}, \cite[Conjecture~1]{Cap26a}, and \cite[Conjecture~2.5]{Cap26b}; see also Mestrano and Simpson \cite[Section~12, p.~121]{MS16}.

In this paper, we show that the converse fails in rank four.  We briefly describe the construction (see Section~\ref{sec:construction} for the details).  Let $\Sigma\subset\mathbb{P}^2$ be a very general smooth plane quintic curve, and set $S=\mathcal{O}_\Sigma(1)$, $X=\Sigma^{(2)}$, and $F=S^{[2]}$.  Multiplication on length two subschemes, together with the isomorphism $S^2\cong K_\Sigma$, gives an $\Omega_X^1$-valued symmetric form $\beta: \mathrm{Sym}^2(F)\longrightarrow\Omega_X^1$, or equivalently its adjoint morphism $\psi:F\longrightarrow F^*\otimes\Omega_X^1$, while the evaluation sequence of $F$ produces, for a suitable section $s\in H^0(\Sigma,S)\cong H^0(X,F)$, an extension
\[
 0\longrightarrow F^*\longrightarrow E_s\longrightarrow F
 \longrightarrow0.
\]
The Higgs field $\theta_s:E_s\to E_s\otimes\Omega_X^1$ is the composite
\[
E_s\twoheadrightarrow F\xrightarrow{\psi}
 F^*\otimes\Omega_X^1\hookrightarrow E_s\otimes\Omega_X^1.
\]

Our main result is the following.

\begin{thm}\label{thm:main}
There exist a smooth complex projective surface $X$ and a rank four Higgs
bundle $\mathcal{E}_s=(E_s,\theta_s)$ on $X$ with the following properties:
\begin{itemize}
\item[(1)] For every smooth projective curve $C$ and every morphism
      $f:C\to X$, the pullback Higgs bundle $f^*\mathcal{E}_s$ is semistable.
\item[(2)] The underlying vector bundle $E_s$ satisfies
      \[
       \mathrm{det}(E_s)\cong\mathcal{O}_X,
       \qquad \int_Xc_2(E_s)=10.
      \]
\end{itemize}
In particular, $\Delta(E_s)\neq0$ in $H^4(X,\mathbb{Q})$.
\end{thm}

Part~(1) implies, in particular, that $\mathcal{E}_s$ is semistable with
respect to every polarization on $X$.  Indeed, if $H$ is ample and
$\mathcal{G}\subset E_s$ is a saturated $\theta_s$-invariant subsheaf with
$\mu_H(\mathcal{G})>\mu_H(E_s)$, then the restriction of $\mathcal{G}$ to a
general smooth curve in $|mH|$, for $m\gg0$, would destabilize the
corresponding pullback Higgs bundle, contrary to part~(1).  The main point is
the quantifier in part~(1): for an arbitrary morphism $f:C\to X$, let
$V=f^*F$.  A Higgs subsheaf of $f^*\mathcal{E}_s$ determines subsheaves
$I\subset V$ and $J\subset V^*$ satisfying
$(f^*\psi)(I)\subset J\otimes K_C$, but the image $I$ need not be
saturated, and the corresponding subsheaf exists only when a restriction pushout of the extension class vanishes.  It is therefore essential to retain both the nonsaturated image and this obstruction.

To analyze this criterion, the possible rank drop of the contracted symmetric morphism $u_f:=f^*\psi$ is governed by a divisor in $\mathbb{P}(T_X)$, for $\mathbb{P}(T_X)$ parametrizing one-dimensional subspaces of tangent spaces.  More precisely, contraction of $\beta_x$ with a nonzero tangent vector $v\in T_xX$ gives a symmetric form $\beta_{x,v}\in\mathrm{Sym}^2(F_x^*)$, and the equation $\mathrm{det}(\beta_{x,v})=0$ defines the root divisor of $\beta$.  This divisor is smooth and isomorphic to
$\Sigma\times\Sigma$, and its projection to $X$ is the double cover $\Sigma\times\Sigma\to\Sigma^{(2)}$. Consequently, if $f:C\to X$ is nonconstant and $u_f$ has generic rank one, then $f(C)$ is either a coordinate curve or the diagonal.  Direct degree estimates control the coordinate curves, while on the diagonal the relevant restriction pushout class is represented by an extension whose middle term is a rank two bundle $P_s$ of degree zero, and we choose $s$ so that $P_s$ is semistable.   When $u_f$ has generic rank two, the only remaining case is an isotropic line, and the extension class together with the simplicity of the Jacobian of $\Sigma$ excludes a destabilizing subsheaf.

\begin{rem}\label{rem:history}
The argument for the Chern-class vanishing statement for Higgs-numerically flat bundles in \cite[Corollary~3.2]{BGO07} ultimately uses the base change compatibility for absolute Higgs Grassmannians asserted in \cite[Step 1 in the proof of Theorem 1.3]{BHR06}.  This compatibility fails for the Higgs Grassmannians associated to pullback Higgs bundles in the sense of \eqref{eq:pullback-higgs-field}, since contraction by $df$ may produce new invariant quotients.  For example, let $X=A\times B$ be a product of elliptic curves, let $0\neq\alpha\in H^0(A,\Omega_A^1)$, and equip $\mathcal{O}_X^2$ with the Higgs field
\[
 \theta=\begin{pmatrix}p_A^*\alpha&0\\0&0\end{pmatrix}.
\]
Its absolute rank one Higgs Grassmannian consists of the two eigenline sections, whereas the pullback Higgs field along a fiber of $p_A$ is zero, so every line quotient is invariant.  Thus the cited base change argument does not apply to condition~(2); our proof uses neither Higgs numerical flatness nor this compatibility and checks every pullback Higgs bundle directly.  The rank two result cited above is proved by a separate argument and is unaffected.

We also point out that \cite[Introduction]{BLG16} states that counterexamples had been found and refers to unpublished work that later appeared as \cite{BBG19}.  The latter paper instead states the same assertion as an open conjecture and does not contain a counterexample.
\end{rem}

The paper is organized as follows.  In Section \ref{sec:construction}, we give the explicit construction of the Higgs bundle and compute its Chern classes.  In Section \ref{sec:curves}, we give an exact criterion for Higgs subsheaves of pullback Higgs bundles, classify the curves for which the contracted morphism $u_f=f^*\psi$ has generic rank one, and choose the extension class by studying its restriction to the diagonal. Finally, in Section \ref{sec:semistability}, we prove the semistability of every pullback Higgs bundle and complete the proof of Theorem \ref{thm:main}.

\vspace{5mm}

\textbf{Acknowledgements}.
The author is grateful to Ugo Bruzzo and Beatriz Gra\~na Otero
for their interest in this work and for their helpful comments
on the terminology. The author would also like to acknowledge the assistance of OpenAI's ChatGPT (version 5.6 Plus) in performing repeated calculations and testing a wide range of examples. This process eventually led to the counterexample constructed in this paper and to the observation that the plane quintic is the lowest-degree plane curve for which the present tautological construction can work. ChatGPT also assisted the author in revising the manuscript into its present form.

\section{The explicit construction}\label{sec:construction}

In this section, we first recall two results about plane quintics that will be used in the construction.  The simplicity assertion is a special case of Ciliberto and van der Geer \cite[Corollary~1.2]{CvG92}, while the description of the degree four pencils is Noether's theorem \cite[Theorem~2.1]{Har02} (see also \cite[Theorem~2.3.1]{Nam84}).

\begin{lem}\label{lem:plane-quintic}
A very general smooth plane quintic has a simple Jacobian.  Every smooth plane quintic $\Sigma$ has gonality four, and every base-point-free line bundle $L\in\mathrm{Pic}^4(\Sigma)$ is of the form $L\cong\mathcal{O}_\Sigma(1)(-p)$ for a unique point $p\in\Sigma$, with $|L|$ given by the pencil induced by projection from $p$.
\end{lem}

\begin{proof}
The first assertion follows by applying \cite[Corollary 1.2]{CvG92} to $\mathbb{P}^2$ with the very ample line bundle $\mathcal{O}_{\mathbb{P}^2}(5)$.  For the remaining assertions, let $L\in\mathrm{Pic}^4(\Sigma)$ be base-point-free, a general
two-dimensional subspace of $H^0(\Sigma,L)$ is base-point-free and defines a minimal pencil, whereas every minimal pencil on a smooth plane quintic is induced by projection from a point of the curve by Noether's theorem \cite[Theorem 2.1]{Har02}.  Consequently, $L\cong\mathcal{O}_\Sigma(1)(-p)$ and $h^0(\Sigma,L)=2$, and the point $p$ is unique because $\mathcal{O}_\Sigma(p-p')\cong\mathcal{O}_\Sigma$ for $p,p'\in\Sigma$ implies $p=p'$.
\end{proof}

Fix a very general smooth plane quintic
$\Sigma\subset\mathbb{P}^2$ as in Lemma \ref{lem:plane-quintic}, and let $S=\mathcal{O}_\Sigma(1)$ and $X=\Sigma^{(2)}$, so that $S^2\cong K_\Sigma$.  Let $F=S^{[2]}$ be the rank two tautological bundle on $X$, whose fiber over a length two effective divisor $Z$ is $F_Z=H^0(Z,S|_Z)$, and denote $H^0(\Sigma,S)$ and $\mathrm{det}(F)$ by $W$ and $N$, respectively.  Choose a nonzero element of $\mathrm{det}(W)$, which identifies the constant line bundle $\mathrm{det}W\otimes\mathcal{O}_X$ with $\mathcal{O}_X$.  We use the standard tautological bundle convention associated to the universal divisor; see \cite[Section 1.2]{Kru20}.  The Chern numbers in the following lemma also follow from \cite[Lemmas 2.3 and 2.4]{Kru25}, but we include the direct geometric argument because it also identifies the secant morphism and proves the ampleness of $N$.

\begin{lem}\label{lem:tautological}
There is a short exact sequence
\begin{align}\label{eq:evaluation}
 0\longrightarrow N^{-1}\longrightarrow W\otimes\mathcal{O}_X
 \longrightarrow F\longrightarrow0.
\end{align}
Moreover, $F$ is nef, $N$ is ample, $N^2=10$, and $\int_Xc_2(F)=10$.
\end{lem}

\begin{proof}
Since $S$ separates length two subschemes, its evaluation morphism is surjective, and taking determinants after the preceding trivialization identifies its rank one kernel with $N^{-1}$, which proves \eqref{eq:evaluation}.

Consider the secant morphism
\begin{align*}
   \varphi:X&\longrightarrow(\mathbb{P}^2)^*,\\
 p+q&\longmapsto\overline{pq}, 
\end{align*}
where $2p$ is mapped to the tangent line $T_p\Sigma$.  For a line $\ell$, the fiber over $\ell$ parametrizes the length two subschemes of the Cartier divisor $\Sigma\cap\ell$.  If $p$ occurs in this divisor with multiplicity $m$, then, after choosing a uniformizer, its local ring at $p$ is
$\mathcal{O}_{\Sigma,p}/I_{\Sigma\cap\ell,p}\cong\mathbb{C}[t]/(t^m)$, which contains a unique subscheme of length $k$ supported at $p$ for each $0\leq k\leq m$.  Since there are only finitely many distributions of total length two among the support points of $\Sigma\cap\ell$, every fiber has finitely many closed points and is therefore zero-dimensional.  Thus
$\varphi$ is quasi-finite, and its properness implies that it is finite.  A transverse line meets $\Sigma$ in five distinct points, and hence $\mathrm{deg}\varphi=\binom{5}{2}=10$.

Let $Q$ be the universal rank two quotient bundle on
$(\mathbb{P}^2)^*$.  The fiberwise description gives
$F\cong\varphi^*Q$ and $N\cong\varphi^*\mathcal{O}(1)$, so $F$ is globally generated and hence nef, while $N$ is ample and $N^2=10$.  Since $c_2(Q)$ is the class of a point, one also obtains $\int_Xc_2(F)=\mathrm{deg}\varphi=10$.
\end{proof}

For general discriminant formulas for tautological bundles on symmetric
products of curves, see \cite[Theorem~1.1]{Kru25}.

The cotangent bundle of $X$ admits the tautological description
$\Omega_X^1\cong(K_\Sigma)^{[2]}$.  Indeed, for every $Z\in X$ one has
$T_ZX=H^0(\mathcal{O}_Z(Z))$, whose dual is identified with
$H^0(K_\Sigma|_Z)$ by the residue pairing, and these identifications
globalize over the universal divisor.  Multiplication on each length two
subscheme, together with $S^2\cong K_\Sigma$, therefore defines an $\Omega_X^1$-valued symmetric form $\beta$, or equivalently its adjoint morphism $\psi$:
\begin{align}
 \beta&:\mathrm{Sym}^2(F)\longrightarrow\Omega_X^1,
 \label{eq:beta}\\
 \psi&:F\longrightarrow F^*\otimes\Omega_X^1,
 \qquad \psi(v)(w)=\beta(v,w).
 \label{eq:psi}
\end{align}

After trivializing $S$ on a length two subscheme $Z$, the fiber of $\beta$ is the multiplication map
$\mathrm{Sym}^2(H^0(\mathcal{O}_Z))\to H^0(\mathcal{O}_Z)$, which is surjective because multiplication by $1$ reaches every element. Consequently, $\beta$ is pointwise surjective and its transpose
\begin{align}\label{eq:tangent-injection}
 \beta^\vee:T_X\longrightarrow\mathrm{Sym}^2(F^*)
\end{align}
is pointwise injective.

For $s\in W$, let $\alpha_s=\beta(s,s)\in H^0(X,\Omega_X^1)$, and let $a:X\to J(\Sigma)$ be a degree two Abel map.  Then $\alpha_s=a^*\lambda_s$, where $\lambda_s$ is the translation invariant one form on $J(\Sigma)$ corresponding to $s^2\in H^0(\Sigma,K_\Sigma)$, because the differential of $a$ is dual to the restriction morphism $H^0(K_\Sigma)\to H^0(K_\Sigma|_Z)$.

We now construct the extension, using the canonical isomorphism
$F^*\otimes N\cong F$ for the rank two bundle $F$, together with the identities $H^0(X,F)=W$ and $H^0(X,F^*)=0$.  Indeed, the ampleness of $N$ gives $H^0(X,N^{-1})=0$, while Serre duality and Kodaira vanishing give $H^1(X,N^{-1})^*=H^1(X,K_X\otimes N)=0$, so taking cohomology in
\eqref{eq:evaluation} yields $H^0(X,F)=W$.  On the other hand, a section of $F^*$ would determine a fixed functional in $W^*$ annihilating the kernel of $W\to F_Z$ for every $Z\in X$; these kernel lines are represented by the equations of the secant and tangent lines to $\Sigma$ and span $W$, so this functional must vanish.

Applying $\mathrm{Hom}(\bullet,F^*)$ to \eqref{eq:evaluation} and using $H^0(X,F^*)=0$, we obtain an injective connecting homomorphism
\[
 \delta:W=H^0(X,F^*\otimes N)
 \longrightarrow\mathrm{Ext}^1_X(F,F^*).
\]
For a nonzero section $s\in W$, let $e_s=\delta(s)$ and denote the corresponding extension by
\begin{align}\label{eq:extension}
 0\longrightarrow F^*\longrightarrow E_s\longrightarrow F
 \longrightarrow0.
\end{align}
Equivalently, this extension is the pushout of \eqref{eq:evaluation} by the homomorphism $N^{-1}\to F^*$ corresponding to $s$.

The adjoint morphism $\psi$ in \eqref{eq:psi} defines the Higgs field $\theta_s:E_s\longrightarrow E_s\otimes\Omega_X^1$ as the composite
\begin{align}\label{eq:higgs-field}
 E_s\twoheadrightarrow F\xrightarrow{\psi}
 F^*\otimes\Omega_X^1\hookrightarrow E_s\otimes\Omega_X^1.
\end{align}
Since the image of $\theta_s$ is contained in $F^*\otimes\Omega_X^1$ and its restriction to $F^*$ vanishes, one has $\theta_s\wedge\theta_s=0$, and hence $\mathcal{E}_s=(E_s,\theta_s)$ is a Higgs bundle.

\begin{prop}\label{prop:chern}
The determinant and the second Chern class of $E_s$ satisfy
\[
 \mathrm{det}(E_s)\cong\mathcal{O}_X,
 \qquad \int_Xc_2(E_s)=10.
\]
In particular, $\Delta(E_s)\neq0$ in $H^4(X,\mathbb{Q})$.
\end{prop}

\begin{proof}
Taking determinants in \eqref{eq:extension} gives
$\mathrm{det}(E_s)\cong\mathrm{det}(F^*)\otimes\mathrm{det}(F)
\cong\mathcal{O}_X$, while the Whitney formula and
Lemma~\ref{lem:tautological} give
\[
 \int_Xc_2(E_s)
 =\int_X\bigl(c_2(F^*)+c_1(F^*)c_1(F)+c_2(F)\bigr)
 =10-10+10=10.
\]
Since $X$ is a connected smooth projective surface,
$H^4(X,\mathbb{Q})\cong\mathbb{Q}$, so $c_2(E_s)$ is nonzero, which implies $\Delta(E_s)\neq0$.
\end{proof}

\section{Geometry of pullback Higgs bundles}\label{sec:curves}

In this section, we establish the criterion for Higgs subsheaves of pullback Higgs bundles, describe the root divisor associated to the $\Omega_X^1$-valued symmetric form $\beta$ in \eqref{eq:beta}, and choose the extension class by analyzing its restriction to the diagonal.

\subsection{The curve criterion}

Let $C$ be a smooth projective curve and let $f:C\to X$ be a
morphism.  Set
\[
 V=f^*F,\qquad f^*\beta: \mathrm{Sym}^2(V)\longrightarrow K_C,\qquad
 u_f=f^*\psi:=(1\otimes df)\circ f^*(\psi):
 V\longrightarrow V^*\otimes K_C.
\]
Thus $f^*\beta$ is the symmetric form corresponding by adjunction to $u_f$. The underlying vector bundle $f^*E_s$ is an extension of $V$ by $V^*$, and the pullback Higgs field $f^*\theta_s$ is the composite
\[
 f^*E_s\twoheadrightarrow V\xrightarrow{\,u_f\,}
 V^*\otimes K_C\hookrightarrow f^*E_s\otimes K_C.
\]
Let $e_f=f^*e_s\in\mathrm{Ext}^1_C(V,V^*)$ denote the class of this extension.  For subsheaves $I\subset V$ and $J\subset V^*$, denote the natural inclusion and quotient by $i_I:I\hookrightarrow V$ and $q_J:V^*\twoheadrightarrow V^*/J$, respectively.

\begin{lem}\label{lem:curve-criterion}
Let $I\subset V$ and $J\subset V^*$ be subsheaves.  There exists a Higgs
subsheaf $H$ of the pullback Higgs bundle $f^*\mathcal{E}_s$ with
$J=H\cap V^*$ and
$I=\mathrm{Im}(H\to V)$ if and only if
\begin{align}\label{eq:invariance}
 u_f(I)\subset J\otimes K_C
\end{align}
and the restriction pushout class
\begin{align}\label{eq:obstruction}
 \epsilon_{I,J}:=(q_J)_*(i_I^*e_f)\in\mathrm{Ext}^1_C(I,V^*/J)
\end{align}
vanishes.  In this case,
\[
 0\longrightarrow J\longrightarrow H\longrightarrow I
 \longrightarrow0,
 \qquad
 \mathrm{deg} H=\mathrm{deg} I+\mathrm{deg} J.
\]
If $H$ is saturated in $f^*E_s$, then $J$ is saturated in $V^*$, whereas
$I$ need not be saturated in $V$.
\end{lem}

\begin{proof}
Suppose first that $H$ exists.  Intersecting it with the kernel $V^*$ and
mapping it to the quotient $V$ gives a commutative diagram
\[
\begin{tikzcd}
0\arrow[r]&J\arrow[r]\arrow[d]&H\arrow[r]\arrow[d]
&I\arrow[r]\arrow[d]&0\\
0\arrow[r]&V^*\arrow[r]&f^*E_s\arrow[r]&V\arrow[r]&0.
\end{tikzcd}
\]
Since $f^*\theta_s$ vanishes on $V^*$ and factors through
$u_f:V\to V^*\otimes K_C$, the condition that $H$ being preserved by $f^*\theta_s$ is equivalent to \eqref{eq:invariance}. Moreover, a lift of $I$ whose intersection with $V^*$ is $J$ exists precisely when the pullback of the bottom extension to $I$, followed by pushout to $V^*/J$, splits, which is equivalent to \eqref{eq:obstruction}, while the top row gives the degree formula. Conversely, the vanishing in \eqref{eq:obstruction} gives the required subsheaf, and \eqref{eq:invariance} makes it a Higgs subsheaf.  Its saturation is again a Higgs subsheaf because over a discrete valuation ring, the inclusion of $t^mx$ in the subsheaf implies that $t^m(f^*\theta_s)(x)=(f^*\theta_s)(t^mx)$ lies in its tensor product with $K_C$, and saturation removes the factor $t^m$. Moreover, saturation can only increase the degree.  Finally, if $H$ is saturated, then the natural injection $V^*/J\longrightarrow f^*E_s/H$ shows that $V^*/J$ is
torsion-free and hence that $J$ is saturated, whereas no analogous conclusion holds for $I$.
\end{proof}

\begin{rem}
The image $I$ in Lemma~\ref{lem:curve-criterion} need not be saturated. Indeed, let $R$ be a discrete valuation ring with uniformizer $t$, fix an integer $c\geq1$, and consider
\[
 V=Rv_1\oplus Rv_2,\qquad V^*=Ra_1\oplus Ra_2,
 \qquad u(v_1)=a_2,\quad u(v_2)=a_1.
\]
Inside the split extension $V^*\oplus V$, endowed with the square zero Higgs field that vanishes on $V^*$ and is induced by $u$, the submodule $H=Ra_2+R(a_1+t^cv_1)$ is saturated and preserved by the Higgs field, but its image in $V$ is
$t^cRv_1$. Thus replacing $I$ by its saturation would lose genuine Higgs subsheaves and the associated colength data.
\end{rem}

Since $\mathrm{deg} f^*E_s=0$, it is enough to prove
$\mathrm{deg}H\leq0$ for every proper saturated Higgs subsheaf $H$.  The bundle $V$ is nef, so every quotient line bundle of $V$ has nonnegative degree and every line subsheaf of $V^*$ has nonpositive degree.  The possible rank pairs $(\mathrm{rk}I,\mathrm{rk}J)$ associated to a proper nonzero saturated Higgs subsheaf $H$ are
\[
 (0,1),(1,0),(0,2),(1,1),(2,0),(1,2),(2,1).
\]
If $\mathrm{rk}I=0$, then $\mathrm{deg}J\leq0$, while
$\mathrm{rk}J=2$ and $\mathrm{rk}I\leq1$ imply
$\mathrm{deg}I+\mathrm{deg}J\leq\mathrm{deg}V-\mathrm{deg}V=0$.  The pair $(2,0)$ is impossible when $u_f\neq0$, and when $u_f$ has generic rank two, the same is true for $(1,0)$ and $(2,1)$.  Consequently, only $(1,1)$ must be considered in generic rank two, whereas in generic rank one the additional pairs are $(1,0)$ and $(2,1)$.

\subsection{The root divisor and the generic rank one case}

We now classify the generic rank one curves.  Away from the diagonal, the natural decompositions $F_{p+q}=S_p\oplus S_q$ and
$T_{p+q}X=T_p\Sigma\oplus T_q\Sigma$ show that contraction by a tangent vector $(v_p,v_q)$ gives the diagonal symmetric form
\begin{align}\label{eq:off-diagonal-form}
 (z_p,z_q)\longmapsto v_pz_p^2+v_qz_q^2.
\end{align}

For $x\in X$ and $0\neq v\in T_xX$, let
\[
 \beta_{x,v}(a,b)=\langle\beta_x(a,b),v\rangle,
 \qquad a,b\in F_x,
\]
and denote its adjoint by $\psi_{x,v}=(1\otimes v)\circ\psi_x:F_x\to F_x^*$.  Since $\det(\psi_{x,\lambda v})=\lambda^2\det(\psi_{x,v})$, the equation $
 \det(\psi_{x,v})=\det(\beta_{x,v})=0$ defines a divisor $\mathscr{R}\subset\mathbb{P}(T_X)$, which we call the
root divisor of $\beta$.

\begin{lem}\label{lem:root-divisor}
The divisor $\mathscr{R}$ is isomorphic to $\Sigma\times\Sigma$, and under this isomorphism the projection $\mathscr{R}\to X$ is the quotient map $\pi:\Sigma\times\Sigma\to\Sigma^{(2)}$.  Moreover, if $f: C\to X$ is a nonconstant morphism from a smooth projective curve and $u_f$ has generic rank one, then $f(C)=X_p:=\{p+q \mid q\in\Sigma\}$ for some $p\in\Sigma$, or $f(C)=j(\Sigma):=\{2q \mid q\in\Sigma\}$.
\end{lem}

\begin{proof}
Consider the morphism
$\rho:\Sigma\times\Sigma\longrightarrow\mathbb{P}(T_X)$ that sends $(p,q)$ to the tangent direction at $p+q$ obtained by varying $q$ while keeping $p$ fixed; this direction remains nonzero along the diagonal.  Formula \eqref{eq:off-diagonal-form} identifies $\Sigma\times\Sigma\setminus\Delta$ with $\mathscr{R}$ over $X\setminus j(\Sigma)$, so it remains only to examine the points above the diagonal.

Choose a local coordinate $z$ on $\Sigma$, let $\sigma_1=z_1+z_2$ and $\sigma_2=z_1z_2$, and use $1,z$ as a basis of the length two algebra $\mathcal{O}[z]/(z^2-\sigma_1z+\sigma_2)$.  For a tangent vector $\xi$ with $a=d\sigma_1(\xi)$ and $b=d\sigma_2(\xi)$, the contracted
symmetric form $\beta_{x,\xi}$ has the matrix
\[
 \begin{pmatrix}
 a&\sigma_1a-b\\
 \sigma_1a-b&(\sigma_1^2-\sigma_2)a-\sigma_1b
 \end{pmatrix}
\]
whose determinant is
\[
 -(b-z_1a)(b-z_2a).
\]
No nonzero degenerate direction has $a=0$, and hence the affine coordinate $t=b/a$ covers $\mathscr{R}$, where its equation becomes $t^2-\sigma_1t+\sigma_2=0$.  The inverse coordinates $z_1=t$ and $z_2=\sigma_1-t$ identify $\mathscr{R}$ locally with
$\Sigma\times\Sigma$; these local inverses are restrictions of $\rho$, so they glue and also prove smoothness above the diagonal.

On the open subset of $C$ where $df\neq0$, the tangent direction of $f$ defines a lift to $\mathscr{R}$, which extends across the zeros of $df$ because $C$ is smooth and $\mathscr{R}$ is proper over $X$.  Let $g=(g_1,g_2):C\to\Sigma\times\Sigma$ be the resulting morphism, so that $f=\pi\circ g$.  If $g(C)$ is not contained in the diagonal, choose bases $\partial$ of $T_C$ and $e_i$ of $g_i^*T_\Sigma$ at the generic point and
define $g_i'\in K(C)$ by $dg_i(\partial)=g_i'e_i$.  Formula
\eqref{eq:off-diagonal-form} gives $g_1'g_2'=0$, so one differential vanishes identically and the corresponding map is constant in characteristic zero, which shows that $f(C)$ is a coordinate curve.  If $g(C)$ is contained in the diagonal, then $f(C)=j(\Sigma)$.
\end{proof}

\subsection{The diagonal restriction and the choice of the extension class}

We now choose the section $s$ used in \eqref{eq:extension}.  Restricting the evaluation sequence \eqref{eq:evaluation} to the diagonal $j:\Sigma\to X$ and using the natural jet filtration gives
\begin{align}\label{eq:jets}
 0\longrightarrow S^3=S\otimes K_\Sigma
 \longrightarrow j^*F=J^1S\longrightarrow S\longrightarrow0
\end{align}
while the corresponding inverse-image sequence is the twisted conormal sequence
\begin{align}\label{eq:conormal}
 0\longrightarrow S^{-4}\longrightarrow
 M_S:=\Omega_{\mathbb{P}^2}^1(1)|_\Sigma
 \longrightarrow S^3\longrightarrow0.
\end{align}
The latter is the inverse image of the subline $S^3\subset j^*F$ in \eqref{eq:jets} under $W\otimes\mathcal{O}_\Sigma\to j^*F$.    For completeness, the first sequence also follows by pulling the universal divisor back to $\Sigma\times\Sigma$ and restricting along the diagonal, where one has
\[
 0\longrightarrow\mathcal{O}_{\Delta}(-\Delta)
 \longrightarrow\mathcal{O}_{2\Delta}
 \longrightarrow\mathcal{O}_{\Delta}\longrightarrow0,
 \qquad
 \mathcal{O}_{\Delta}(-\Delta)\cong K_\Sigma.
\]
In particular, $j^*N^{-1}\cong S^{-4}$, and under this isomorphism the composite
$j^*N^{-1}\to j^*F^*=(j^*F)^*\twoheadrightarrow S^{-3}$ obtained by restricting the defining morphism $N^{-1}\to F^*$ of
\eqref{eq:extension} is multiplication by $s$.  Indeed, the
canonical isomorphism $F\cong F^*\otimes N$ sends $v$ to the functional $w\mapsto w\wedge v$, so the map $N^{-1}\to F^*$ corresponding to $s$ sends a local section $\eta$ to the functional $w\mapsto\eta(w\wedge s)$.  Its restriction to the subline $S^3\subset j^*F$ is the image of $s$ in the quotient
$j^*F/S^3\cong S$, and hence, under $\mathrm{Hom}(S^{-4},S^{-3})\cong H^0(\Sigma,S)$, the composite corresponds
to the original section $s$.

The restriction of $j^*e_s$ to $S^3\subset j^*F$, followed by pushout through $(j^*F)^*\twoheadrightarrow S^{-3}$, is the extension
\[
 0\longrightarrow S^{-3}\longrightarrow P_s
 \longrightarrow S^3\longrightarrow0
\]
obtained by pushing out \eqref{eq:conormal} by multiplication by $s$. Equivalently, it fits into the pushout diagram
\[
\begin{tikzcd}
0\arrow[r]&S^{-4}\arrow[r]\arrow[d,"s"]
&M_S\arrow[r]\arrow[d]&S^3\arrow[r]\arrow[d,equal]&0\\
0\arrow[r]&S^{-3}\arrow[r]&P_s\arrow[r]&S^3\arrow[r]&0.
\end{tikzcd}
\]
If the zero divisor $H_s=(s=0)$ is reduced, one also has
\begin{align}\label{eq:elementary-modification}
 0\longrightarrow M_S\longrightarrow P_s
 \longrightarrow S^{-3}|_{H_s}\longrightarrow0.
\end{align}
At every point $r\in H_s$, the added direction is the conormal line $(S^{-4})_r\subset(M_S)_r$.  Over the local discrete valuation ring $R=\mathcal{O}_{\Sigma,r}$, choose a parameter $t$ and a basis $e_0,e_1$ of $M_S$ such that $s=t$ and $e_0$ generates the subline $S^{-4}\subset M_S$; if $b$ is a generator of the added line, then the elementary modification is the lattice inclusion
\[
 Re_0\oplus Re_1
 =R(tb)\oplus Re_1
 \subset Rb\oplus Re_1=P_s.
\]

\begin{lem}\label{lem:good-section}
For every $s$ in a nonempty Zariski open subset of $W$, the divisor $H_s$ is reduced and the bundle $P_s$ is semistable of degree zero.
\end{lem}

\begin{proof}
Let $L\subset M_S$ be a saturated line subbundle.  Since
$M_S\subset W\otimes\mathcal{O}_\Sigma$ is a subbundle and
$H^0(\Sigma,M_S)=0$ by the Euler sequence, $L^{-1}$ is a nontrivial globally
generated line bundle; the fact that $\Sigma$ has gonality four therefore
gives
\begin{align}\label{eq:line-degree}
 \mathrm{deg} L\leq-4.
\end{align}
Lemma~\ref{lem:plane-quintic} shows that equality holds precisely for
\[
 L=L_p:=S^{-1}(p),\qquad p\in \Sigma.
\]

The corresponding embedding can be described explicitly.  For the line
bundle $S(-p)$, the multiplication map
\[
 W\otimes H^0(S(-p))\longrightarrow H^0(K_\Sigma(-p))
\]
has a one-dimensional kernel, generated in coordinates with
$p=[0:0:1]$ by $x\otimes y-y\otimes x$, and hence the inclusion
$L_p\to M_S$ is unique up to a scalar.  In the fiber
$\mathbb{P}(M_{S,r})$, identified with the pencil of lines through $r$, its
value is $\overline{pr}$ for $r\neq p$ and its limiting value at $r=p$ is
$T_p\Sigma$, whereas the conormal direction at $r$ is represented by
$T_r\Sigma$.

Suppose that $R\subset P_s$ is a saturated line subbundle of positive
degree, let $R_0=R\cap M_S$, and let $L$ be its saturation in $M_S$.  The
elementary modification \eqref{eq:elementary-modification} gives
$\mathrm{length}(R/R_0)\leq5$, which, together with
\eqref{eq:line-degree}, forces
\[
 \mathrm{deg} R=1,\quad \mathrm{deg} L=-4,\quad R_0=L=L_p,\quad
 \mathrm{length}(R/R_0)=5.
\]
The last equality means that $L_p$ agrees with the added conormal direction
at all five points of $H_s$, and the preceding fiberwise description shows
that this is equivalent to
\begin{align}\label{eq:bad-line}
 p\in T_r\Sigma\qquad\text{for every }r\in H_s.
\end{align}

Let $U\subset(\mathbb{P}^2)^*$ be the open set of lines transverse to
$\Sigma$, and for $p\in\Sigma$ let
$R_p=\{r\in\Sigma \mid p\in T_r\Sigma\}$.  Since $R_p$ is cut out by the nonzero
first polar quartic of $\Sigma$ with respect to $p$, it is finite.  Consider
$\mathcal{B}=\{(p,\ell)\in\Sigma\times U \mid \Sigma\cap\ell\subset R_p\}$ and
the finite \'etale universal divisor $\mathcal{D}\to U$.  On
$\mathcal{D}_\Sigma=(\Sigma\times U)\times_U\mathcal{D}$, the complement of
the polar condition is open, and its image in $\Sigma\times U$ is open
because the morphism is finite \'etale, so $\mathcal{B}$ is closed.  For a
fixed $p$, every bad transverse line is determined by two of its five
intersection points, all belonging to the finite set $R_p$; hence the
fibers of $\mathcal{B}\to\Sigma$ are finite and
$\mathrm{dim}\,\mathcal{B}\leq1$.  Since $\Sigma$ is projective, the image
of $\mathcal{B}$ in $U$ is closed of dimension at most one, and a line
outside this image determines a section $s$ for which $H_s$ is reduced and
\eqref{eq:bad-line} never occurs.  Thus $P_s$ has no line
subbundle of positive degree and, since $\mathrm{deg}P_s=0$, it is semistable.
\end{proof}

From now on, fix a section $s$ as in Lemma~\ref{lem:good-section}.

\begin{lem}\label{lem:finite-pullback}
Let $C$ be a smooth projective curve and let $h:C\to\Sigma$ be a
finite morphism.  Then $h^*P_s$ is semistable of degree zero.
\end{lem}

\begin{proof}
Let $0=P_0\subset\cdots\subset P_k=P_s$ be a Jordan--H\"older filtration.
By the Narasimhan--Seshadri correspondence \cite{NS65}, its stable degree zero factors correspond to irreducible unitary representations, whose pullbacks
by $h$ are polystable of degree zero. Therefore, pulling back the filtration
shows that $h^*P_s$ is semistable, regardless of whether $h$ is ramified.
\end{proof}

\section{Proof of the main theorem}\label{sec:semistability}

We prove part (1) of Theorem~\ref{thm:main}, first for curves on which $u_f$ has generic rank two.

\begin{prop}\label{prop:rank-two-curves}
Let $C$ be a smooth projective curve and let $f:C\to X$ be a
nonconstant morphism such that $u_f$ has generic rank two.  Then
the pullback Higgs bundle $f^*\mathcal{E}_s$ is semistable.
\end{prop}

\begin{proof}
By the discussion after Lemma~\ref{lem:curve-criterion}, only the pair $(\mathrm{rk}I,\mathrm{rk}J)=(1,1)$ needs to be
considered.  Let $\bar I\subset V$ and $\bar J\subset V^*$ be the
saturations of $I$ and $J$, and let $\bar I^\perp\subset V$ be the saturated orthogonal line for the symmetric form $u_f$.  Then
\[
 \bar J=(V/\bar I^\perp)^*,\qquad
 \mathrm{deg}\bar J=\mathrm{deg}\bar I^\perp-\mathrm{deg} V.
\]
If $\bar I\neq\bar I^\perp$, their wedge is nonzero, and hence
\[
 \mathrm{deg} I+\mathrm{deg} J
 \leq\mathrm{deg}\bar I+\mathrm{deg}\bar I^\perp-\mathrm{deg} V\leq0.
\]

Suppose that $\bar I=\bar I^\perp$, and let $Q_C=V/\bar I$ and
$q=\mathrm{deg}Q_C$.  The bundle $V$ is nef, so $q\geq0$, while
$\bar J=Q_C^*$ and $\mathrm{deg}\bar J=-q$.  Pulling \eqref{eq:evaluation} back to $C$, define $M_C=\mathrm{ker}(W\otimes\mathcal{O}_C\to Q_C)$, which fits into the exact
sequence
\begin{align}\label{eq:MC}
 0\longrightarrow(\mathrm{det} V)^{-1}\longrightarrow M_C
 \longrightarrow\bar I\longrightarrow0.
\end{align}
Let $\bar s\in H^0(C,Q_C)$ be the image of $f^*s$.  The restriction of $f^*e_s$ to $\bar I$, followed by pushout to
$V^*/\bar J=\bar I^*$, is the pushout
\[
 0\longrightarrow\bar I^*\longrightarrow G_{C,I,s}
 \longrightarrow\bar I\longrightarrow0
\]
of \eqref{eq:MC} by $\bar s$.
Indeed, $M_C$ is the inverse image of $\bar I$ under
$W\otimes\mathcal{O}_C\to V$, while the canonical isomorphism
$\mathrm{Hom}((\mathrm{det} V)^{-1},\bar I^*)\cong Q_C$ sends the
homomorphism induced by $f^*s$ to its image $\bar s\in H^0(C,Q_C)$; hence this is the claimed restriction pushout identification.

Suppose first that $\bar s\neq0$.  Its zero divisor has length $q$, and there is an elementary modification
\begin{align}\label{eq:GC-modification}
 0\longrightarrow M_C\longrightarrow G_{C,I,s}
 \longrightarrow\bar I^*|_{(\bar s=0)}\longrightarrow0.
\end{align}
Over a local discrete valuation ring $R$, choose a generator $e_0$ of the first term in \eqref{eq:MC} and a generator $e$ of $\bar I^*$.  If $\bar s=t^m$, then $G_{C,I,s}=(M_C\oplus Re)/(e_0-t^me)$, so the local colength is exactly $m$.  Since $J\subset\bar J$ and $\epsilon_{I,J}=0$, the extension restricted to $I$ splits after pushout to $V^*/J$, and hence also after the further pushout to $V^*/\bar J$; the latter is the restriction of $G_{C,I,s}$ and gives an injection $I\to G_{C,I,s}$ whose image we denote by $L$.  Intersecting with $M_C$ in \eqref{eq:GC-modification} gives $L/(L\cap M_C)\lhook\joinrel\longrightarrow G_{C,I,s}/M_C$, and hence
$\mathrm{deg}(L\cap M_C)\geq\mathrm{deg} I-q$. But $M_C\subset W\otimes\mathcal{O}_C$, so every line subsheaf of $M_C$
has nonpositive degree.  Therefore $\mathrm{deg} I\leq q$ and
$\mathrm{deg} J\leq-q$, so $\mathrm{deg} I+\mathrm{deg} J\leq0$.

Suppose instead that $\bar s=0$.  Then the section $f^*s$ of $V$ lies in the isotropic line $\bar I$, and hence $f^*\alpha_s=\langle u_f(f^*s),f^*s\rangle=0$. Choose base points and factor $a\circ f:C\to J(\Sigma)$ through the induced
homomorphism $J(C)\longrightarrow J(\Sigma)$. The equality $(a\circ f)^*\lambda_s=0$ says that $\lambda_s$ vanishes on
the image abelian subvariety.  Since $J(\Sigma)$ is simple and
$\lambda_s\neq0$, this image is zero.  Thus $a\circ f$ is constant.  The degree two Abel map is injective because $\Sigma$ is nonhyperelliptic, so $f$ is constant, a contradiction.
\end{proof}

Now we treat the coordinate curves in Lemma~\ref{lem:root-divisor}.  Let $i_p:\Sigma\to X$ be given by $q\mapsto p+q$.

\begin{lem}\label{lem:coordinate}
Let $C$ be a smooth projective curve and let $h:C\to\Sigma$ be a
finite morphism.  Then the pullback Higgs bundle $(i_p\circ h)^*\mathcal{E}_s$ is semistable.
\end{lem}

\begin{proof}
There is an exact sequence
\begin{align}\label{eq:coordinate-sequence}
 0\longrightarrow S_p\otimes\mathcal{O}_\Sigma(-p)
 \longrightarrow i_p^*F\longrightarrow S\longrightarrow0.
\end{align}
This is the case $n=2$ of the standard pullback sequence for tautological bundles \cite[Proposition~1.5]{Kru20}.  It also follows directly from the ordered square description of the universal divisor: the constant summand $S_p\otimes\mathcal{O}_\Sigma$ acquires one negative elementary modification when the moving point reaches $p$. For $q\in\Sigma\setminus\{p\}$, use the fiberwise decomposition
$(i_p^*F)_q=S_p\oplus S_q$.  Formula \eqref{eq:off-diagonal-form} shows that contraction with the tangent direction of the moving point is represented by a diagonal matrix with entries $0$ and a nonzero local function.  Its saturated kernel on $\Sigma$ is therefore $S_p\otimes\mathcal{O}_\Sigma(-p)$, while its saturated image in $(i_p^*F)^*$ is the dual $S^{-1}$ of the quotient in
\eqref{eq:coordinate-sequence}.
Let $m=\mathrm{deg} h$.  The saturated kernel of $u_{i_p\circ h}$ has degree $-m$, the quotient line has degree $5m$, the saturated image line in $V^*$ has degree $-5m$, and $\mathrm{deg} V=4m$.  Ramification introduces only a common scalar zero factor in $u_{i_p\circ h}$ and does not change its
saturated generic kernel or image.

For the pair $(1,0)$, condition \eqref{eq:invariance} forces $I$ into the kernel, so $\mathrm{deg} I\leq-m$.  For the pair $(2,1)$, condition \eqref{eq:invariance} forces the saturation of $J$ to be the image line.  Since $J$ is saturated by Lemma \ref{lem:curve-criterion}, it equals the saturated image line, and hence $\mathrm{deg} I+\mathrm{deg} J\leq4m-5m<0$. For the pair $(1,1)$, if $I$ lies in the kernel, then $\mathrm{deg} I\leq-m$ and $\mathrm{deg} J\leq0$.  Otherwise, $I$ maps nontrivially to the degree $5m$ quotient and the saturation of $J$ is the degree $-5m$ image line.  Thus $\mathrm{deg} I+\mathrm{deg} J\leq0$ in every case.  All other rank pairs were already controlled after Lemma~\ref{lem:curve-criterion}.
\end{proof}

It remains to consider the diagonal.

\begin{lem}\label{lem:diagonal}
Let $C$ be a smooth projective curve and let $h:C\to\Sigma$ be a
finite morphism.  Then the pullback Higgs bundle $(j\circ h)^*\mathcal{E}_s$ is semistable.
\end{lem}

\begin{proof}
Let $m=\mathrm{deg} h$.  After pulling back \eqref{eq:jets}, denote the resulting subbundle and quotient by $K=h^*S^3$ and $Q=h^*S$, respectively.
Then
\[
 \mathrm{deg} K=15m,\qquad\mathrm{deg} Q=5m,\qquad\mathrm{deg} V=20m.
\]
In the local coordinates used in the proof of Lemma \ref{lem:root-divisor}, a tangent vector $v$ along the diagonal has $a=2v$ and $b=2zv$.  The matrix in that proof becomes
\[
 2v\begin{pmatrix}1&z\\ z&z^2\end{pmatrix}.
\]
Its kernel is the jet subline $S\otimes K_\Sigma=S^3$ in
\eqref{eq:jets}, and its image is the dual $S^{-1}$ of the quotient $S$. After pullback by $h$, ramification multiplies the matrix by a common local equation of an effective divisor and therefore does not change its saturated generic kernel or image; these are $K$ and $J_0=Q^*\subset V^*$, respectively, with $\mathrm{deg}J_0=-5m$. The restriction of $(j\circ h)^*e_s$ to $K$, followed by pushout through $V^*\twoheadrightarrow K^*$, is represented by the pullback of the defining extension of $P_s$, whose middle term is $h^*P_s$, which is semistable of degree zero by Lemma~\ref{lem:finite-pullback}.

For the pair $(1,0)$, condition \eqref{eq:invariance} gives $I\subset K$.  If the obstruction vanishes, pushout by $V^*\twoheadrightarrow K^*$ lifts $I$ to a line subsheaf of $h^*P_s$, hence $\mathrm{deg} I\leq0$.

For the pair $(2,1)$, condition \eqref{eq:invariance} gives the same generic line for $J$ and $J_0$.  Both are saturated in $V^*$, so $J=J_0$.  Let $I_K=I\cap K=K(-Z_K)$ and $c=\mathrm{length}(V/I)$.  The injection $K/(I\cap K)\longrightarrow V/I$ gives $\mathrm{deg} Z_K\leq c$. Restriction to $I_K$ and pushout to $K^*=V^*/J_0$ lift $I_K$ to $h^*P_s$. Thus $15m-\mathrm{deg} Z_K\leq0$, and consequently $c\geq15m$.  It follows that $\mathrm{deg} I+\mathrm{deg} J\leq(20m-c)-5m\leq0$.

Consider the pair $(1,1)$.  If $I$ is not generically contained in $K$, then $I\to Q$ is nonzero and the saturation of $J$ is $J_0$.  Therefore $\mathrm{deg} I\leq5m$ and $\mathrm{deg} J\leq-5m$. If $I\subset K$ and $J$ is not generically contained in $J_0$, then $J\to K^*$ is nonzero, and hence $\mathrm{deg} I\leq15m$ and $\mathrm{deg} J\leq-15m$.
Finally, if $I\subset K$ and $J\subset J_0$, the same restriction pushout argument lifts $I$ to $h^*P_s$.  Thus $\mathrm{deg} I\leq0$, while $\mathrm{deg} J\leq0$.  This proves the assertion.
\end{proof}

\begin{proof}[Proof of Theorem~\ref{thm:main}]
Choose $\Sigma$ and $s$ as in Lemmas~\ref{lem:plane-quintic} and
\ref{lem:good-section}, respectively.  The Higgs bundle $\mathcal{E}_s$ is defined
by \eqref{eq:extension} and \eqref{eq:higgs-field}.

Let $C$ be a smooth projective curve and let $f:C\to X$ be a
morphism.  If $f$ is constant, then $df=0$, and the pullback Higgs bundle $f^*\mathcal{E}_s$ is a trivial bundle of degree zero with zero Higgs field. Suppose that $f$ is nonconstant.  By the pointwise injectivity of \eqref{eq:tangent-injection}, $u_f$ is nonzero at the generic point.  If it has generic rank two, semistability follows from Proposition \ref{prop:rank-two-curves}.  If it has generic rank one, Lemma \ref{lem:root-divisor} shows that $f=i_p\circ h$ or $f=j\circ h$ for a finite morphism $h: C\to\Sigma$. The conclusion then follows from Lemma \ref{lem:coordinate} or Lemma \ref{lem:diagonal}.

Thus the pullback Higgs bundle $f^*\mathcal{E}_s$ is semistable for every $f$.  On the other hand, Proposition \ref{prop:chern} gives $\mathrm{det}(E_s)\cong\mathcal{O}_X$ and $\int_Xc_2(E_s)=10$, so $\Delta(E_s)\neq0$ in $H^4(X,\mathbb{Q})$.
\end{proof}

\begin{cor}\label{cor:counterexample}
Bruzzo--Gra\~na Otero's curve semistability conjecture for Higgs bundles is false.  
\end{cor}

\bigskip

\noindent\small{\textsc{School of Mathematics, Nanjing University}\\
Nanjing 210093, China}\\
\emph{E-mail address}: \texttt{pfhwangmath@gmail.com}

\bigskip

\end{document}